%% file: Magnetic8.tex
\documentclass[a4paper,reqno,11pt]{amsart}
\usepackage{amssymb,amsmath,amsthm}
\usepackage{amsfonts}
\usepackage{fullpage}
\usepackage{mathrsfs, lmodern}
\usepackage[font=small]{caption}
\usepackage{color}
\definecolor{darkred}{rgb}{0.6,0.2,0.2}
\usepackage[colorlinks=true,linkcolor=blue, citecolor=darkred]{hyperref}
\usepackage{xcolor, tcolorbox}
\usepackage{stackrel, enumerate}
\usepackage{graphicx, bbm}
\newcommand\bA{{\bf A}}
\definecolor{darkblue}{rgb}{0.2,0.2,0.6}

\input commands.tex

\usepackage[left=2.7cm, right=2.7cm, marginparwidth=2cm, textheight = 24cm ]{geometry}
\usepackage[normalem]{ulem}
\definecolor{DarkGreen}{rgb}{0,0.5,0.1}

\definecolor{DarkBlue}{rgb}{0,0.1,0.5}

\newcommand{\curl}{\mathrm{curl}\,}
\newcommand\soutD{\bgroup\markoverwith
	{\textcolor{DarkGreen}{\rule[.5ex]{2pt}{1pt}}}\ULon}
\newcommand\soutP{\bgroup\markoverwith
	{\textcolor{blue}{\rule[.5ex]{2pt}{1pt}}}\ULon}
\newcommand{\Hm}[1]{\leavevmode{\marginpar{\tiny%
			$\hbox to 0mm{\hspace*{-0.5mm}$\leftarrow$\hss}%
			\vcenter{\vrule depth 0.1mm height 0.1mm width \the\marginparwidth}%
			\hbox to
			0mm{\hss$\rightarrow$\hspace*{-0.5mm}}$\\
			\relax\raggedright #1}}}

\begin{document}

\title[]{Inequalities between Dirichlet and Neumann eigenvalues of the magnetic Laplacian}

\author[V. Lotoreichik]{Vladimir Lotoreichik}
\address[V. Lotoreichik]{Department of Theoretical Physics, Nuclear Physics Institute, 	Czech Academy of Sciences, Hlavn\'{i} 130, 25068 \v Re\v z, Czech Republic}
\email{lotoreichik@ujf.cas.cz}

\subjclass{35P15, 81Q10}

\keywords{magnetic Laplacian, homogeneous magnetic field, inequalities between Dirichlet and Neumann eigenvalues, min-max principle}

\maketitle

\begin{abstract}
	We consider the magnetic Laplacian with the homogeneous magnetic field in two and three dimensions.
	We prove that 
	the $(k+1)$-th magnetic Neumann eigenvalue of a bounded convex planar domain is not larger than its $k$-th magnetic Dirichlet eigenvalue. In three dimensions, we restrict our attention to convex domains, which are invariant under rotation by an angle of $\pi$ around an axis parallel to the magnetic field. For such domains, we prove that
	the $(k+2)$-th magnetic Neumann eigenvalue is not larger than the $k$-th magnetic Dirichlet eigenvalue
	provided that this Dirichlet eigenvalue is simple.
	The proofs rely on a modification of the strategy due to Levine and Weinberger. 
\end{abstract}
\section{Introduction}
The aim of the present paper is to obtain inequalities between Dirichlet and Neumann eigenvalues of the magnetic Laplacian on a bounded  domain in the spirit of the  Levine-Weinberger and Friedlander-Filonov inequalities  for the usual (non-magnetic) Laplacian. The Levine-Weinberger inequality~\cite{LW86} states that on a $d$-dimensional ($d\ge 2$) bounded convex domain the $(k+d)$-th Neumann eigenvalue of the Laplacian does not exceed its $k$-th Dirichlet eigenvalue. Without the convexity assumption, it was proved by Friedlander~\cite{F91} that the $(k+1)$-th Neumann eigenvalue of the Laplacian does not exceed its $k$-th Dirichlet eigenvalue. This inequality was later shown by Filonov~\cite{F05} to be strict and under weaker regularity assumptions on the domain. A discussion of these results can be found in the recent monograph~\cite{LMP}. 

Our analysis is motivated by the question to what extent these eigenvalue inequalities survive the presence of the
homogeneous magnetic field. We focus on physically motivated two- and three-dimensional settings.
Inequalities between Dirichlet and Neumann eigenvalues of the magnetic Laplacian are already considered by Frank and Laptev in~\cite{FL10}. The present paper is partially inspired by their considerations and contains improvements of their results under additional assumptions. 

First, it is shown in~\cite{FL10} in two dimensions that,
below any Landau level,
the number of magnetic Neumann eigenvalues  exceeds at least by one the number of magnetic Dirichlet eigenvalues.
This result suggests that an inequality between individual eigenvalues with an index shift can hold. In the present paper, we prove 
that on any  bounded planar convex domain the $(k+1)$-th Neumann eigenvalue of the magnetic Laplacian does not exceed its $k$-th Dirichlet eigenvalue. 
It is also proved in~\cite{FL10} 
that in three dimensions the \mbox{$(k+1)$-th} magnetic Neumann eigenvalue is smaller than the $k$-th magnetic Dirichlet eigenvalue. A natural question is whether one can improve the index shift in this inequality.
We restrict our attention to convex domains, which coincide with itself upon rotation by an angle of $\pi$ around an axis parallel to the vector of the magnetic field. For such symmetric domains, we prove that the
$(k+2)$-th Neumann eigenvalue of the magnetic Laplacian does not exceed its $k$-th Dirichlet eigenvalue provided that the latter Dirichlet eigenvalue is simple.

Recently, Rohleder~\cite{R23} established, using a new variational principle, that in two dimensions the Levine-Weinberger inequality is valid for all sufficiently smooth bounded simply-connected domains (without the convexity assumption). The validity of the Levine-Weinberger inequality for general bounded domains in all space dimensions is a prominent open problem~\cite[Conjecture 3.2.41]{LMP}. In this perspective, one may ask whether the inequalities obtained in the present paper are also valid without the convexity assumption.
\subsection{Main results}
Let us describe our results in more detail. To this aim we need first to define the operators we consider. 
Let $\Omg\subset\dR^d$ with $d \in\{ 2,3\}$ be a bounded Lipschitz
domain. Note that any bounded convex domain has Lipschitz boundary (see~\cite[Corollary 1.2.2.3]{Gr}).  We choose the vector potential ${\bf A} \colon\Omg\arr\dR^d$ of the homogeneous magnetic field of intensity $b > 0$ in the Landau gauge
\[	
\begin{aligned}
	{\bf A}(x) =
	\begin{cases} 
	(0,bx_1)^{\top},& d = 2,\\
	(0,b x_1,0)^\top,&d = 3,
	\end{cases}
\end{aligned}	  
\]
where we use the convention $x = (x_1,x_2)$ for $d = 2$ and $x = (x_1,x_2,x_3)$ for $d = 3$. 
In two dimensions, the respective magnetic field is given by $\curl {\bf A} =
\p_1 A_2 - \p_2 A_1 = b$ and it points in the direction orthogonal to the plane. The particular choice of the gauge of the homogeneous magnetic field
in two dimensions is not 
restrictive for further considerations, because our results concern convex domains, for which there is gauge invariance of the spectra for the magnetic Laplacians; \cf~\cite[Appendix D]{FH}.
In three dimensions,
we get after a simple computation
that the magnetic field is given by $\curl {\bf A} = (0,0,b)^\top$ and thus points along the $x_3$-axis. 
The direction of the vector of the magnetic field in three dimensions is essential in our analysis.

Let $H^1(\Omg)$ be the standard $L^2$-based first-order Sobolev space and let $H^1_0(\Omg)$ be the closure of $C^\infty_0(\Omg)$ in $H^1(\Omg)$.
Consider the following closed, densely defined, non-negative quadratic forms
\begin{equation}\label{eq:forms}
\begin{aligned}
	\frh_{b,{\rm D}}^\Omg[u] 
	& := 
	\|(\nb - \ii {\bf A})u\|^2_{L^2(\Omg;\dC^d)},\qquad \dom\frh_{b,{\rm D}}^\Omg := H^1_0(\Omg),\\
	\frh_{b,{\rm N}}^\Omg[u]& := 
		\|(\nb - \ii {\bf A})u\|^2_{L^2(\Omg;\dC^d)},\qquad \dom\frh_{b,{\rm N}}^\Omg := H^1(\Omg),
\end{aligned}	
\end{equation}
in the Hilbert space $L^2(\Omg)$; \cf~\cite[Section 1.2]{FH}.
Let $\sfH_{b,{\rm D}}^\Omg$ and $\sfH_{b,{\rm N}}^\Omg$ be the self-adjoint operators in $L^2(\Omg)$ associated with the forms~$\frh_{b,{\rm D}}^\Omg$ and~$\frh_{b,{\rm N}}^\Omg$, respectively. Compactness of the embedding of $H^1(\Omg)$ into $L^2(\Omg)$ yields that the spectra of both $\sfH_{b,\rm D}^\Omg$ and $\sfH_{b,\rm N}^\Omg$ are purely discrete.
Denote by $\{\lm_k^b(\Omg)\}_{k\in\dN}$
and $\{\mu_k^b(\Omg)\}_{k\in\dN}$ the eigenvalues
of, respectively, $\sfH_{b,\rm D}^\Omg$ and $\sfH_{b,\rm N}^\Omg$, enumerated in the non-decreasing order and repeated with multiplicities taken into account.

It was proved in two dimensions by Frank and Laptev in~\cite[Theorem 3.1\,(1)]{FL10} that
\[
	\#\big\{k\in\dN\colon \lm_k^b(\Omg) \le b(2q+1)\big\} +1 \le \#\big\{k\in\dN\colon \mu_k^b(\Omg) < b(2q+1)\big\},\qquad \text{for all}\,\,q\in\dN_0.
\]
Namely, for any $q\in\dN_0$, if the operator $\sfH_{b,\rm D}^\Omg$ has $j\in\dN_0$ eigenvalues less or equal to $b(2q+1)$, then the operator $\sfH_{b,\rm N}^\Omg$ has at least $j+1$ eigenvalues less than $b(2q+1)$. As we have already mentioned, this result suggests that an inequality with an index shift between the individual eigenvalues of $\sfH_{b,\rm D}^\Omg$ and $\sfH_{b,\rm N}^\Omg$ may hold. This inequality under an additional convexity assumption is the first main result of the present paper.
\begin{thm}\label{thm:main1}
	Let $\Omg\subset\dR^2$ be a bounded, convex domain. Then, for any $b > 0$,
	\[
		\mu_{k+1}^b(\Omg) \le \lm_k^b(\Omg),\qquad\text{for all}\,\, k\in\dN.
	\]
\end{thm}
In the same paper, Frank and Laptev
established in~\cite[Theorem 3.1\,(2)]{FL10} that, in three dimensions the inequality $\mu_{k+1}^b(\Omg)< \lm_k^b(\Omg)$ holds for any $b > 0$ and all $k\in\dN$.
In our second main result, we prove a (non-strict) eigenvalue inequality with a higher index shift under the convexity assumption in the class of domains invariant under rotation by angle of $\pi$ around the $x_3$-axis (the axis along which the magnetic field is oriented). In order to formulate this result we introduce the mapping
\begin{equation}\label{eq:J}
	\sfJ \colon\dR^3\arr\dR^3,\qquad \sfJ x = (-x_1,-x_2,x_3)^\top.
\end{equation}	

\begin{thm}\label{thm:main2}
	Let $\Omg\subset\dR^3$ be a bounded, convex domain such that $\sfJ(\Omg) = \Omg$. Let $b  > 0$ be fixed. Let $k\in\dN$ be such that $\lm_k^b(\Omg)$ is a simple eigenvalue of $\sfH_{b,\rm D}^\Omg$. Then it holds that
	\[
	\mu_{k+2}^b(\Omg) \le \lm_k^b(\Omg).
	\]
\end{thm}
Theorem~\ref{thm:main2} applies to all domains rotationally invariant around the $x_3$-axis including the ball or properly oriented ellipsoid with two semi-axes of equal length. It also applies to some symmetric polyhedral domains. For example, it can be applied to the octahedron oriented so that one of its diagonals lies on the $x_3$-axis.  

The Levine-Weinberger inequality is sharp in the sense that, in general, the index $k+d$ can not be replaced by $k+d+1$. Indeed, as is well-known, for the disk the fourth Neumann eigenvalue of the Laplacian is larger than its first Dirichlet eigenvalue; \cf~\cite[Example~4.6]{R23}. We provide a numerical evidence in Section~\ref{sec:disk} below, that in the magnetic case the third Neumann eigenvalue of the disk is not larger than its first Dirichlet eigenvalue. Thus, an example showing sharpness of Theorem~\ref{thm:main1} (if it exists) should either be not the disk or should be based on comparison of higher eigenvalues for the disk.
In view of this observation, it remains an open question whether the index $k+1$ in Theorem~\ref{thm:main1} can be replaced by $k+2$. The assumption of simplicity of the $k$-th Dirichlet eigenvalue in Theorem~\ref{thm:main2} is technical and we expect that it can be removed. However, the symmetry of the domain is essential for the argument and it is not {\it a priori} clear whether the inequality in Theorem~\ref{thm:main2} holds for general convex domains. 

Both theorems are first proved for convex polyhedral domains.
The results for general convex domains are deduced via approximation. 
The eigenvalue inequalities for polyhedral domains are proved via the min-max principle, where
as the trial subspace we select the span of the first $k$ magnetic Dirichlet eigenfunctions and of all but one first-order partial derivatives of the $k$-th Dirichlet eigenfunction.
In the original proof~\cite{LW86} of the Levine-Weinberger inequality for the Laplacian it was possible to use all first-order partial derivatives of the $k$-th Dirichlet eigenfunction in the construction of the trial subspace.
In the magnetic case, we need to exclude the partial derivative with respect to the variable $x_1$, on which  the vector potential $\bA$ depends. Similar idea was used in~\cite{LR17} for the Laplacian with mixed boundary conditions and 
in~\cite{R21} for the Schr\"odinger operator on a  bounded domain with (electric) potential dependent on the number of directions strictly less than the space dimension. The main novel difficulties of the magnetic problem is that the perturbation is non-additive and that the eigenfunctions are genuinely complex-valued. The fact that the eigenfunctions of the magnetic Dirichlet Laplacian are complex-valued leads to additional difficulties. It is the main reason for the symmetry assumption imposed in the formulation of Theorem~\ref{thm:main2}.  

\subsection{Overview of the literature}
The results of Levine, Weinberger, Friedlander, Filonov, and Rohleder mentioned in the beginning of the introduction were preceded by eigenvalue inequalities in more special settings or under additional assumptions.
P\'{o}lya proved
in~\cite{P52} that the second Neumann eigenvalue
of the Laplacian on a bounded smooth planar domain 
is smaller than its first Dirichlet eigenvalue. For convex, two-dimensional domains with $C^2$-boundary Payne proved in~\cite{P55} that the $(k+2)$-th Neumann eigenvalue of the Laplacian is smaller than its $k$-th Dirichlet eigenvalue, the inequality which was later generalized by Levine and Weinberger to higher dimensions. The Friedlander-Filonov inequality was conjectured by Payne in~\cite{P91} and preceded by the paper of Aviles~\cite{A86}, in which this inequality was proved under an additional assumption on the mean curvature of the boundary.
  
Similar type inequalities are also obtained for other types of differential operators.  Besides the already mentioned papers:  ~\cite{FL10} on the sub-Laplacian on the Heisenberg group and the magnetic Laplacian,~\cite{LR17} on the Laplacian with mixed boundary conditions,~\cite{R21} on the Schr\"odinger operator, we also refer in this context to recent results~\cite{DT22} on the Stokes operator
and to \cite{L24, Pr19} on the biharmonic and more general polyharmonic operators. The literature on this subject is quite extensive and it is difficult to mention all related results.

Recently, geometric bounds~\cite{CLPS23, EKP16, KL23, LS15} and isoperimetric inequalities~\cite{CLPS24, FFGKS23} for the eigenvalues of the magnetic Neumann Laplacian attracted considerable attention. 
The eigenvalues inequalities obtained in the present paper can also be viewed as estimates of magnetic Neumann eigenvalues but
in terms of magnetic Dirichlet eigenvalues.
\subsection*{Structure of the paper}
In Section~\ref{sec:pre}, we collect properties of the magnetic Laplacian on a bounded domain and recall a formula for integration by parts on polyhedral domains. Theorem~\ref{thm:main1} is proved in Section~\ref{sec:proof1}.
In turn, a proof of Theorem~\ref{thm:main2} is provided in Section~\ref{sec:proof2}. Section~\ref{sec:disk} contains
a numerical analysis of the case of the disk, which is motivated
by the intention to check sharpness of the index shift in Theorem~\ref{thm:main1}. The paper is complemented by Appendix~\ref{app}, in which
we show lower semicontinuity of magnetic Neumann eigenvalues on varying convex domains satisfying an inclusion property. This auxiliary statement is used in extending our results from polyhedral domains to general convex domains.  
\section{Preliminaries}\label{sec:pre}
\subsection{Magnetic Laplacian on a bounded domain}
Let $\Omg\subset\dR^d$, $d\in\{2,3\}$, be a bounded Lipschitz domain. We denote by $\p\Omg$ the boundary of $\Omg$ and by $\nu$ the outer unit normal vector to $\p\Omg$. This normal vector is well defined almost everywhere on $\p\Omg$ thanks to Rademacher's theorem.
By the magnetic Laplacian we understand the following
second-order differential expression
\[
	-(\nb -\ii\bA)^2 
	= -\Delta  + 2\ii b x_1\p_2  + b^2 x_1^2.
\]
It is a special case of the most general second-order  differential expression considered in~\cite[Chapter 4]{McL} (see also~\cite{BR12, BR20}).  

The magnetic Dirichlet Laplacian $\sfH_{b,\rm D}^\Omg$ associated with the quadratic form~$\frh_{b,\rm D}^\Omg$ in~\eqref{eq:forms} can be characterised as (see \cite[Section 1.2]{FH})
\begin{equation}\label{key}
	\sfH_{b,\rm D}^\Omg u = -(\nb - \ii\bA)^2 u,\qquad \dom\sfH_{b,\rm D}^\Omg = \big\{u\in H^1_0(\Omg)\colon \Delta u \in L^2(\Omg)\big\}.
\end{equation}
It follows directly from~\cite[Theorem 3.2.1.2]{Gr} that
\begin{equation}\label{eq:convex}
	\Omg\,\, \text{is convex}\qquad \Longrightarrow\qquad \dom\sfH_{b,{\rm D}}^\Omg = H^2(\Omg)\cap H^1_0(\Omg).
\end{equation}
Recall that for any $u\in H^1(\Omg)$ its trace on the boundary $u|_{\p\Omg}$ is well defined as a function in $H^{1/2}(\p\Omg)$ and for $u\in  H^1(\Omg)$ with
$\Delta u\in L^2(\Omg)$ its trace of the normal derivative $\p_\nu u|_{\p\Omg}$ is well defined as a distribution in $H^{-1/2}(\p\Omg)$; see~\cite{BGM22} for details.

The magnetic Neumann Laplacian $\sfH_{b,\rm N}^\Omg$ associated with the quadratic form~$\frh_{b,\rm N}^\Omg$ in~\eqref{eq:forms} can be characterised as (see \cite[Section 1.2]{FH} and~\cite[Section 5]{BR20})
\begin{equation}\label{key}
\begin{aligned}	
	\sfH_{b,\rm N}^\Omg u &= -(\nb - \ii\bA)^2 u,\\ \dom\sfH_{b,\rm N}^\Omg &= \big\{u\in H^1(\Omg)\colon
	\Delta u\in L^2(\Omg),\,
	\nu\cdot(\nb -\ii\bA)u = 0~~\text{on}~\p\Omg\big\},
\end{aligned}	
\end{equation}
where $\nu\cdot(\nb -\ii\bA)u$ on the boundary $\p\Omg$
should be identified with $\p_\nu u|_{\p\Omg} - \ii \nu\cdot\bA|_{\p\Omg} u|_{\p\Omg}$ being, by the above discussion, well defined
as a distribution
in $H^{-1/2}(\p\Omg)$.
Recall also that for any $u\in H^1(\Omg)$ and any $v\in H^1(\Omg)$ with $\Delta v \in L^2(\Omg)$ we have by~\cite[Eq. (5.112)]{BGM22} the following formula for integration by parts
\begin{equation}\label{eq:integrationbyparts}
	((\nb -\ii \bA)u, (\nb -\ii \bA) v)_{L^2(\Omg;\dC^d)}
	\!=\!
	(u,-(\nb -\ii \bA)^2 v)_{L^2(\Omg)}
	\!+\!\langle u|_{\p\Omg},\p_\nu v|_{\p\Omg} - \ii\nu\cdot\bA|_{\p\Omg}v|_{\p\Omg} \rangle_{\frac12,-\frac12},
\end{equation}
where $\langle\cdot,\cdot\rangle_{1/2,-1/2}$ stands for the duality product between $H^{1/2}(\p\Omg)$ and $H^{-1/2}(\p\Omg)$, which is compatible with the inner product in $L^2(\p\Omg)$ (linear in the first entry).

The eigenvalues of the magnetic Laplacians 
$\sfH_{b,\rm D}^\Omg$ and $\sfH_{b,\rm N}^\Omg$
can be characterised by the min-max principle (see, \eg,~\cite[Theorem 1.28]{FLW23})
\begin{subequations}\label{eq:minmax}
\begin{equation}\label{eq:minmaxD}
	\lm_k^b(\Omg)  = \min_{\begin{smallmatrix} \cL\subset H^1_0(\Omg)\\ \dim\cL = k\end{smallmatrix}}\max_{u\in\cL\sm \{0\}}\frac{\|(\nb - \ii \bA)u\|^2_{L^2(\Omg;\dC^d)}}{\|u\|^2_{L^2(\Omg)}},\qquad k\in\dN,
\end{equation}
\begin{equation}\label{eq:minmaxN}
	\mu_k^b(\Omg)  = \min_{\begin{smallmatrix} \cL\subset H^1(\Omg)\\ \dim\cL = k\end{smallmatrix}}\max_{u\in\cL\sm \{0\}}\frac{\|(\nb - \ii \bA)u\|^2_{L^2(\Omg;\dC^d)}}{\|u\|^2_{L^2(\Omg)}},\qquad k\in\dN,
\end{equation}
\end{subequations}
where the minima are taken with respect to $k$-dimensional linear subspaces $\cL$ of respective Sobolev spaces and the maxima are taken with respect to all non-zero functions in the subspace  $\cL$. The minima are attained on the spans of $k$ orthonormal eigenfunctions of  respective operators corresponding to their first $k$ eigenvalues while the maximum is subsequently attained on the eigenfunctions corresponding to the $k$-th eigenvalue.
As a direct consequence of the above variational characterisation we get that for any $k\in\dN$ there exists a subspace $\cL\subset H^1_0(\Omg)$ with $\dim\cL = k$ such that
\begin{equation}\label{eq:Lk}
	\frh_{b,{\rm D}}^\Omg[u] \le \lm_k^b(\Omg)\|u\|^2_{L^2(\Omg)},\qquad\text{for all}\,\,u\in\cL.
\end{equation}
The next lemma is a special case of~\cite[Proposition 2.5]{BR12}.
\begin{lem}\label{lem:BR}
	Let $\Omg\subset\dR^d$, $d\in\{2,3\}$, be a bounded connected Lipschitz domain.
	Let an open subset $\omg\subset\p\Omg$ be non-empty. 
	Let $u\in H^2(\Omg)\cap H^1_0(\Omg)$ be such that $-(\nb -\ii\bA)^2u = \lm u$ with some $\lm\in\dR$
	and that $\p_\nu u|_{\omg} = 0$. Then there holds $u\equiv0$.  
\end{lem}
\begin{remark}
Under the assumption $u\in H^2(\Omg)\cap H^1_0(\Omg)$ the Neumann trace $\p_\nu u|_{\p\Omg}$,
is more regular and belongs to $L^2(\p\Omg)$ (see~\cite[Eq. (1.18)]{BGM22}). Thus, its restriction $\p_\nu u|_{\omg}$ to $\omg\subset\p\Omg$ is naturally defined.
\end{remark}
\subsection{Integration by parts in convex polygons and polyhedra}
In this subsection we recall a formula for integration by parts on polyhedral domains, which is derived in~\cite[Appendix A]{LR17} in all space dimensions. Its two-dimensional version  
can be found in the classical monograph by Grisvard (see~\cite[Lemma 4.3.1.2]{Gr}). Before we proceed, we recall the definition of polyhedral domains in $\dR^d$ with $d=2,3$. The definition can be extended to dimensions $d > 3$, but it is not needed for our purposes.
\begin{dfn}
Let  $\Omg\subset\dR^d$, $d\in\{2,3\}$, be a bounded, connected Lipschitz domain.
\begin{myenum}
\item For $d = 2$, $\Omega$ is a polyhedral (or polygonal) domain if its boundary $\p\Omg$ is the union of finitely many line segments.
\item For $d = 3$, $\Omega$ is a polyhedral domain if for each plane $H \subset\dR^3$ the intersection $H \cap\Omg$  is either a polygon in $\dR^2$
(where we identify $H$ with $\dR^2$) or empty.
\end{myenum}
\end{dfn}
We remark that any polyhedral domain is Lipschitz.
Now we are ready to state a formula for integration by parts on convex polyhedral domains. In the below proposition and in the following, we will occasionally use the abbreviation for the partial derivatives $\p_{ij}u = \frac{\p^2u}{\p x_i\p x_j}$ for $1\le i,j\le d$. We will also use the abbreviation $\p_i^2 u = \frac{\p^2 u}{\p x_i^2}$ for $1\le i\le d$.
\begin{prop}[{\cite[Lemma A.1]{LR17}}]
	\label{prop:byparts}
	Let $\Omg\subset\dR^d$, $d \in\{2,3\}$, be a bounded convex polyhedral domain and let $u\in H^2(\Omg)\cap H^1_0(\Omg)$ be real-valued. Then 
	\[
		\int_\Omg (\p_{km} u)(\p_{kj} u)\dd x = 
		\int_\Omg (\p_{mj} u)(\p_{kk} u)\dd x
	\]
	for all $j,k,m\in\{1,2,3\}$ if $d = 3$ or for all $j,k,m\in\{1,2\}$ if $d = 2$.  
\end{prop}
A direct consequence of the above proposition will be useful in the following.
\begin{cor}\label{cor:byparts}
	Let $\Omg\subset\dR^d$, $d \in\{2,3\}$, be a bounded convex polyhedral domain and let $u\in H^2(\Omg)\cap H^1_0(\Omg)$. Then 
	\[
	\Re\left\{\int_\Omg (\p_{km} u)(\ov{\p_{kj} u})\dd x\right\} = 
	\Re\left\{\int_\Omg (\p_{mj} u)(\ov{\p_{kk} u})\dd x\right\}
	\]
	for all $j,k,m\in\{1,2,3\}$ if $d = 3$ or for all $j,k,m\in\{1,2\}$ if $d = 2$.  
\end{cor}
\begin{proof}
	Clearly, we have $\Re u,\Im u \in H^2(\Omg)\cap H^1_0(\Omg)$. By Proposition~\ref{prop:byparts} we get
	\[
	\begin{aligned}
	\int_\Omg \p_{km}(\Re u)\p_{kj}(\Re u)\dd x &= 
	\int_\Omg \p_{mj}(\Re u)\p_{kk} (\Re u)\dd x\\
	\int_\Omg \p_{km}(\Im u)\p_{kj}( \Im u)\dd x &= 
	\int_\Omg \p_{mj}(\Im u)\p_{kk}(\Im u)\dd x,
	\end{aligned}
	\]
	from which the desired identity immediately follows.
\end{proof}
\section{Proof of Theorem~\ref{thm:main1}}\label{sec:proof1}
We will prove the theorem first for convex polygons.
The eigenvalue inequality for general convex domains in $\dR^2$ will be then derived from the eigenvalue inequality for convex polygons via approximation.
Assume until the last step of the proof that $\Omg\subset\dR^2$ is a convex polygon.

\smallskip

\noindent {\it Step 1: construction of a trial subspace of $H^1(\Omg)$.}
For the sake of brevity we use the notation $\lm = \lm_k^b(\Omg)$. By the consequence~\eqref{eq:Lk} of the min-max principle
for the magnetic Laplacian with Dirichlet boundary conditions~\eqref{eq:minmaxD} there exists a subspace
$\cL \subset H^1_0(\Omg)$ with $\dim\cL = k$ such that
\begin{equation}\label{eq:cL}
	\frh_{b,{\rm D}}^\Omg[u] \le \lm \|u\|^2_{L^2(\Omg)},\qquad \text{for all}\,\, u\in\cL.
\end{equation}
Let $v\in H^1_0(\Omg)$ be an eigenfunction of $\sfH_{b,\rm D}^\Omg$ corresponding
to the eigenvalue $\lm$. By~\eqref{eq:convex} we infer that $v\in H^2(\Omg)\cap H^1_0(\Omg)$ and, in particular, we get as a consequence that $\p_2 v\in H^1(\Omg)$. Let us define the subspace $\cM$ of $H^1(\Omg)$ by
\begin{equation}\label{eq:M}
	\cM := \cL + {\rm span}\,\{\p_2 v\}.
\end{equation}
Our next aim is to verify that $\dim\cM = k +1$. To this aim it suffices to show that $\p_2 v\notin H^1_0(\Omg)$.
We will prove this statement by contradiction.
Suppose for the moment that $\p_2 v \in H^1_0(\Omg)$.
Let us choose the side $\G\subset\p\Omg$ of the polygon $\Omg$ such that the (outer) unit normal vector $\nu_\G = (\nu_{\G,1},\nu_{\G,2})^\top$ to this side is not parallel to the $x_1$-axis. It is clearly possible to do, because in any polygon there are sides not parallel to each other. Let us also denote by $\tau_\G = (-\nu_{\G,2},\nu_{\G,1})^\top$ a unit tangential vector to $\G$.  It follows from $v|_{\p\Omg} = 0$ that the tangential derivative of $v$ on $\G$ vanishes
\[
	\tau_\G\cdot(\nb v)|_\G = -\nu_{\G,2}\p_1 v|_\G + \nu_{\G,1}\p_2 v|_\G = -\nu_{\G,2}\p_1 v|_\G =0.
\]
As a consequence we get since $\nu_{\G,2}\ne 0$ that $\p_1 v|_\G = 0$. Altogether, the function $v\in H^2(\Omg)\cap H^1_0(\Omg)$ satisfies the differential equation
$-(\nb - \ii {\bf A})^2 v = \lm v$ in $\Omg$ and its Neumann trace on $\G$ vanishes $\p_\nu v|_{\G} = \nu_{\G,1}\p_1 v|_{\G} + \nu_{\G,2}\p_2 v|_\G =  0$. Hence, we get by Lemma~\ref{lem:BR} that $v$ is identically zero in $\Omg$ leading to a contradiction. 

\smallskip

\noindent {\it Step 2: application of the min-max principle.}  
Any $w\in\cM\subset H^1(\Omg)$ can be represented as
\begin{equation}\label{eq:w}
w = u + c\p_2 v
\end{equation}
with some $c\in\dC$ and $u\in\cL$. Substituting $w$ into the quadratic form for the magnetic Neumann Laplacian on the convex polygon $\Omg$ we get   
\begin{equation}\label{eq:frw}
\begin{aligned}
	\frh_{b,\rm N}^\Omg[w] 
	&= \|(\nb - \ii {\bf A})u\|^2_{L^2(\Omg;\dC^2)}\\
	&\qquad + 2\Re\left\{\left((\nb - \ii {\bf A})u, 
	(\nb - \ii {\bf A})c\p_2 v\right)_{L^2(\Omg;\dC^2)}\right\} 
	+ \|(\nb -\ii {\bf A})c\p_2 v\|^2_{L^2(\Omg;\dC^2)}.
\end{aligned}	
\end{equation}
We will analyse the three terms appearing on the right-hand side in the above equation separately.
We derive from~\eqref{eq:cL} 
\begin{equation}\label{eq:first}
	\|(\nb - \ii {\bf A})u\|^2_{L^2(\Omg;\dC^2)} \le \lm\|u\|^2_{L^2(\Omg)}.
\end{equation}
Note that
\[
	-(\nb - \ii {\bf A})^2\p_2 v 
	= 
	-\Delta \p_2 v + 2\ii b x_1\p_2^2 v + b^2 x_1^2 \p_2 v\\
	= -\p_2(\nb - \ii{\bf A})^2 v = \lm \p_2 v\in L^2(\Omg),
\]
where we first evaluated $\Delta\p_2 v$ in the distributional sense and used that $v$ is an eigenfunction
of $\sfH_{b,\rm D}^\Omg$ corresponding to the eigenvalue $\lm$. In particular, we also get as a consequence
\[
\Delta \p_2 v = 2\ii b x_1\p_2^2 v + b^2x_1^2\p_2 v -\lm\p_2 v\in L^2(\Omg).
\]
By~\eqref{eq:integrationbyparts} we can perform the integration by parts
\begin{equation}\label{eq:second}
	\left((\nb - \ii {\bf A})u, 
	(\nb - \ii {\bf A})c\p_2 v\right)_{L^2(\Omg;\dC^2)}
	=
	-\left(u, 
	(\nb - \ii {\bf A})^2c\p_2 v\right)_{L^2(\Omg)} = 
	\lm(u,c\p_2v)_{L^2(\Omg)}, 
\end{equation}
where the boundary terms vanished since $u\in H^1_0(\Omg)$. 

Next, we decompose $v = v_1 + \ii v_2$ with real-valued
$v_1,v_2\in H^2(\Omg)\cap H^1_0(\Omg)$. Then we get
\begin{equation}\label{eq:p2v}
\begin{aligned}
	\|(\nb -\ii {\bf A})\p_2 v\|^2_{L^2(\Omg;\dC^2)} & =
	\|\nb \p_2 v_1 + {\bf A}\p_2 v_2\|^2_{L^2(\Omg;\dC^2)}
	+ \|\nb \p_2 v_2 - {\bf A}\p_2 v_1\|^2_{L^2(\Omg;\dC^2)}\\
	&=
	\int_\Omg\left(|\p_2^2 v_1 + bx_1\p_2 v_2|^2 + |\p_1\p_2 v_1|^2\right)\dd x\\
	&\qquad+ \int_\Omg \left(|\p_2^2 v_2 -bx_1 \p_2 v_1|^2
	+ |\p_1\p_2 v_2|^2\right)\dd x. 
\end{aligned}
\end{equation}
Since $\Omg$ is a convex polygon and $v_1,v_2\in H^2(\Omg)\cap H^1_0(\Omg)$ are real-valued, we can apply the identity in Proposition~\ref{prop:byparts}
to get
\begin{equation}\label{eq:identity}
\int_{\Omg}|\p_1\p_2 v_j|^2\dd x =  \int_\Omg
\p_1^2 v_j\p_2^2 v_j\dd x,\qquad j=1,2.
\end{equation}
Recall that $v$ is an eigenfunction of $\sfH_{b,\rm D}^\Omg$ corresponding to the eigenvalue $\lm$. Thus, $v$ satisfies the partial differential equation
\begin{equation}\label{eq:v}
	-\Delta v + 2\ii b x_1\p_2 v +b^2 x_1^2 v = \lm v.  
\end{equation}
Using the above decomposition $v = v_1+\ii v_2$ with real-valued $v_1,v_2$ we derive from~\eqref{eq:v}
\begin{equation}
\begin{cases}
-\Delta v_1 - 2bx_1 \p_2 v_2 + b^2x_1^2 v_1 = \lm v_1,\\
-\Delta v_2 + 2bx_1\p_2 v_1 + b^2 x_1^2 v_2 = \lm v_2.
\end{cases}
\end{equation}
From the above two equations we can easily express 
$\p_1^2 v_1$ and $\p_1^2 v_2$  as follows
\begin{equation}\label{eq:v1v2}
\begin{aligned}
	\p_1^2 v_1 &= -\p_2^2 v_1 -2bx_1\p_2 v_2
	+b^2 x_1^2 v_1 - \lm v_1,\\
	\p_1^2 v_2 & = -\p_2^2 v_2 + 2bx_1\p_2 v_1 + b^2x_1^2v_2
	-\lm v_2.
\end{aligned}	
\end{equation}
Combining the identity~\eqref{eq:identity} for $j = 1$ with the first equation in~\eqref{eq:v1v2} we obtain
\begin{equation}\label{eq:p1p2v1}
\begin{aligned}
	\int_\Omg|\p_1\p_2 v_1|^2\dd x 
	&=
	\int_\Omg (-|\p_2^2 v_1|^2 -2bx_1 \p_2 v_2\p_2^2v_1 + 
	b^2 x_1^2 v_1\p_2^2 v_1  -\lm v_1\p_2^2 v_1)\dd x\\
	&=
	\int_\Omg (-|\p_2^2 v_1|^2 -2bx_1 \p_2 v_2\p_2^2v_1 - 
	b^2 x_1^2 |\p_2v_1|^2  +\lm |\p_2v_1|^2)\dd x,
\end{aligned}
\end{equation}
where we performed twice the integration by parts in the second step.
Analogously we find from the identity~\eqref{eq:identity} for $j=2$ combined with the second equation in~\eqref{eq:v1v2}
\begin{equation}\label{eq:p1p2v2}
\begin{aligned}
\int_\Omg|\p_1\p_2 v_2|^2\dd x&=
\int_\Omg\big(-|\p_2^2v_2|^2 + 2b x_1 \p_2 v_1 \p_2^2v_2 +
bx_1^2 v_2\p_2^2 v_2 -\lm v_2\p_2^2 v_2\big)\dd x  \\
&=
\int_\Omg\big(-|\p_2^2v_2|^2 + 2b x_1 \p_2 v_1 \p_2^2v_2 -
b^2x_1^2 |\p_2 v_2|^2 +\lm |\p_2 v_2|^2\big)\dd x. 
\end{aligned}
\end{equation}
It follows from~\eqref{eq:p1p2v1} and~\eqref{eq:p1p2v2}
that
\[
	\int_\Omg\big(|\p_1\p_2 v_1|^2 + |\p_1\p_2 v_2|^2\big)\dd x = \lm\|\p_2 v\|^2_{L^2(\Omg)}
	-\int_\Omg |\p_2^2 v_1 + bx_1\p_2 v_2|^2\dd x -\int_\Omg|
	\p_2^2v_2-bx_1\p_2 v_1|^2\dd x.
\]
Substituting the last identity into~\eqref{eq:p2v} we get
\begin{equation}\label{eq:third}
\|(\nb -\ii {\bf A})\p_2 v\|^2_{L^2(\Omg;\dC^2)}
= \lm\|\p_2 v\|^2_{L^2(\Omg)}.
\end{equation}
Finally, we get
from~\eqref{eq:first},~\eqref{eq:second} and~\eqref{eq:third} that
\[
\begin{aligned}
	\frh_{b,\rm N}^\Omg[w]&\le \lm\|u\|^2_{L^2(\Omg)} + 
	2\lm\Re\big\{(u,c\p_2 v)_{L^2(\Omg)}\big\} + |c|^2\lm\|\p_2 v\|^2_{L^2(\Omg)}\\
	& = \lm\|u+c\p_2 v\|^2_{L^2(\Omg)} = \lm\|w\|^2_{L^2(\Omg)}.
\end{aligned}	
\]
We have shown that for any $w\in\cM\subset H^1(\Omg)$ it holds that $\frh_{b,\rm N}^\Omg[w]\le \lm\|w\|^2_{L^2(\Omg)}$. Since $\cM$ is a $(k+1$)-dimensional linear subspace of $H^1(\Omg)$, we get by the min-max principle~\eqref{eq:minmaxN} that $\mu_{k+1}^b(\Omg)\le \lm = \lm_k^b(\Omg)$. Thus, the inequality in the theorem is proved for convex polygons.

\smallskip

\noindent{\it Step 3: general convex domains.}
For any bounded convex domain $\Omg\subset\dR^2$ one can construct by~\cite{Br08} a sequence of convex polygons $\{\Omg_n\}_{n\in\dN}$ such that the inclusion $\Omg\subset\Omg_n$ holds for all $n\in\dN$ and that $|\Omg_n\sm\Omg|\arr 0$. Let us fix $k\in\dN$.
The min-max principle for the magnetic Laplacian with Dirichlet boundary conditions~\eqref{eq:minmaxD} and the inclusion $\Omg\subset\Omg_n$ immediately imply that
\begin{equation}\label{eq:incl_ineq1}
	\lm_k^b(\Omg) \ge \lm_k^b(\Omg_n),\qquad \text{for all}\,\,n\in\dN.
\end{equation}  
It follows from Proposition~\ref{prop:semicont} in Appendix~\ref{app} that
\begin{equation}\label{eq:incl_ineq2}
	\liminf_{n\arr\infty}\mu_{k+1}^b(\Omg_n) \ge \mu_{k+1}^b(\Omg).
\end{equation}
Combining~\eqref{eq:incl_ineq1} and~\eqref{eq:incl_ineq2} with the eigenvalue inequality in the theorem already shown for convex polygons, we get
\[
	\mu_{k+1}^b(\Omg) \le \liminf_{n\arr\infty}\mu_{k+1}^b(\Omg_n) \le
	\liminf_{n\arr\infty}\lm_{k}^b(\Omg_n) \le \lm_k^b(\Omg),
\]
by which the inequality in the theorem is proved also for general convex domains.

\section{Proof of Theorem~\ref{thm:main2}}\label{sec:proof2}
The proof of this theorem follows similar strategy as the proof of Theorem~\ref{thm:main1}. However, there will also be some new ideas related to the symmetry assumption on the domain. We will prove the theorem first for convex polyhedral domains in $\dR^3$ satisfying the symmetry condition.
The eigenvalue inequality for general convex domains in $\dR^3$ will be then derived via approximation.
Assume until the last step of the proof that $\Omg\subset\dR^3$ is a convex polyhedral domain such that $\sfJ(\Omg) = \Omg$, where the mapping $\sfJ$ is defined as in~\eqref{eq:J}.

\smallskip

\noindent {\it Step 1: Symmetry of the eigenfunction.}
For the sake of brevity we use the notation $\lm = \lm_k^b(\Omg)$ for the $k$-th eigenvalue of $\sfH_{b,\rm D}^\Omg$. Recall that this eigenvalue is assumed to be simple.
Let $v\in H^2(\Omg)\cap H^1_0(\Omg)$ be an eigenfunction of $\sfH_{b,\rm D}^\Omg$ corresponding
to the eigenvalue $\lm$. 
Let us consider the function $\hat{v} := v\circ\sfJ \in H^2(\Omg)\cap H^1_0(\Omg)$.
We get that
\[
\begin{aligned}
	-(\nb - \ii\bA)^2\hat{v} &= -\Delta(v\circ\sfJ) + 2\ii b x_1\p_2(v\circ\sfJ) + b^2x_1^2 (v\circ\sfJ)\\
	&=
	-(\Delta v)\circ\sfJ - 2\ii b x_1(\p_2v)\circ\sfJ + b^2x_1^2 (v\circ\sfJ)\\
	& = \big(-\Delta v + 2\ii b x_1\p_2 v + b^2 x_1^2 v\big)\circ \sfJ = 0.
\end{aligned}	
\]
Hence, $\hat{v}$ is also an eigenfunction of $\sfH_{b,\rm D}^\Omg$ corresponding to the eigenvalue $\lm$. By simplicity of this eigenvalue we conclude that $\hat{v} = z v$ with some $z\in\dC$, $|z| = 1$.

\smallskip
\noindent {\it Step 2: construction of a trial subspace of $H^1(\Omg)$.}
 By the min-max principle
for the magnetic Laplacian with Dirichlet boundary conditions~\eqref{eq:minmaxD} there exists a subspace
$\cL \subset H^1_0(\Omg)$ with $\dim\cL = k$ such that
\begin{equation}\label{eq:cL2}
	\frh_{b,{\rm D}}^\Omg[u] \le \lm \|u\|^2_{L^2(\Omg)},\qquad \text{for all}\,\, u\in\cL.
\end{equation}
For the eigenfunction $v\in H^2(\Omg)\cap H^1_0(\Omg)$
of the magnetic Dirichlet Laplacian on $\Omg$ corresponding to the eigenvalue $\lm$, we get $\p_2 v, \p_3 v\in H^1(\Omg)$. Let us define the subspace $\cM$ of $H^1(\Omg)$ by
\begin{equation}\label{eq:M2}
	\cM := \cL + {\rm span}\,\{\p_2 v,\p_3 v\}.
\end{equation}
Our next aim is to verify that $\dim\cM = k + 2$.
In three dimensions this argument is slightly more involved than in two dimensions. Clearly, we have $\dim\cM \le k+2$. Suppose for the moment that $\dim\cM < k+2$. Hence, there would exist $c_1,c_2\in\dC$ (not both equal to zero) such that $c_1\p_2 v+ c_2\p_3 v\in\cL \subset H^1_0(\Omg)$. Without loss of generality we can assume that real parts not of both $c_1,c_2$ are equal to zero.
Let $\G\subset\p\Omg$ be a face of the polyhedral domain $\Omg$ and let $\tau_1\in\dR^3$ and $\tau_2\in\dR^3$ be  linearly independent unit tangential vectors to $\G$. It is always possible to choose the face $\G$ so that 
\begin{equation}\label{eq:span}
	{\rm span}\,\{\tau_1,\tau_2,(0,c_1,c_2)^\top\} = \dC^3,
\end{equation}
as
otherwise, the non-trivial vector $(0,\Re c_1, \Re c_2)^\top$ would be tangential to every face of the polyhedral domain $\Omg$, which is not possible for clear geometric reasons. Let us assume that the face $\G$ is chosen so that~\eqref{eq:span} holds. It would follow from $v,c_1\p_2 v+c_2\p_3 v\in H^1_0(\Omg)$
that
 \[
	\tau_1\cdot(\nb v)|_\G =  	\tau_2\cdot(\nb v)|_\G = 0,\qquad(0,c_1,c_2)^\top\cdot(\nb v)|_\G = 0.
\]
Hence, we would get using~\eqref{eq:span} that $(\nb v)|_\G = 0$ and, in particular, we would have $\p_\nu v|_\G = 0$.
Thus, the function $v\in H^2(\Omg)\cap H^1_0(\Omg)$ would satisfy the differential equation
$-(\nb - \ii {\bf A})^2 v = \lm v$ in $\Omg$ and $\p_\nu v|_\G = 0$. By Lemma~\ref{lem:BR} we would get that $v$ is identically zero in $\Omg$ leading to a contradiction. 

\smallskip

\noindent {\it Step 3: application of the min-max principle.}  
Any $w\in\cM\subset H^1(\Omg)$ can be represented as
\begin{equation}\label{eq:w}
	w = u + c_1\p_2 v+c_2\p_3 v
\end{equation}
with some $c_1,c_2\in\dC$ and $u\in\cL$. Substituting $w$ into the quadratic form for the magnetic Neumann Laplacian on the domain $\Omg$ we get   
\begin{equation}\label{eq:frw2}
	\begin{aligned}
		\frh_{b,\rm N}^\Omg[w] 
		&= \|(\nb - \ii {\bf A})u\|^2_{L^2(\Omg;\dC^3)}
		 + 2\Re\left\{\left((\nb - \ii {\bf A})u, 
		(\nb - \ii {\bf A})(c_1\p_2 v +c_2\p_3 v )\right)_{L^2(\Omg;\dC^3)}\right\} \\
		&\qquad\qquad\qquad + \|(\nb -\ii {\bf A})(c_1\p_2 v+c_2\p_3 v)\|^2_{L^2(\Omg;\dC^3)}.
	\end{aligned}	
\end{equation}
We will analyse the three terms appearing on the right hand side in the above equation separately.
We derive from~\eqref{eq:cL2} 
\begin{equation}\label{eq:first2}
	\|(\nb - \ii {\bf A})u\|^2_{L^2(\Omg;\dC^3)} \le \lm\|u\|^2_{L^2(\Omg)}.
\end{equation}
Note that
\begin{equation}\label{eq:p2p3v}
\begin{aligned}
&-(\nb - \ii {\bf A})^2(c_1\p_2 v+c_2\p_3 v) \\
&\qquad = 
-\Delta (c_1\p_2 v+c_2\p_3 v)  + 2\ii b x_1\p_2(c_1\p_2 v+c_2\p_3 v) + b^2 x_1^2  (c_1\p_2 v+c_2\p_3 v)\\
&\qquad = -(c_1\p_2+ c_2\p_3)(\nb - \ii{\bf A})^2 v = \lm  (c_1\p_2 v+c_2\p_3 v)\in L^2(\Omg).
\end{aligned}
\end{equation}
where the $\Delta\p_j v$, $j=2,3$, should be understood first in the distributional sense. As a direct consequence of the last computation we infer that $\Delta\p_j v\in L^2(\Omg)$ for $j=2,3$.
By~\eqref{eq:integrationbyparts} we can perform the integration by parts
\begin{equation}\label{eq:second2}
\begin{aligned}	
	&\left((\nb - \ii {\bf A})u, 
	(\nb - \ii {\bf A}) (c_1\p_2 v+c_2\p_3 v)\right)_{L^2(\Omg;\dC^3)}\\
	&\qquad =
	-\left(u, 
	(\nb - \ii {\bf A})^2 (c_1\p_2 v+c_2\p_3 v)\right)_{L^2(\Omg)} = 
	\lm(u, c_1\p_2 v+c_2\p_3 v)_{L^2(\Omg)}, 
\end{aligned}	
\end{equation}
where the boundary terms vanished since $u\in H^1_0(\Omg)$. 

The analysis of the last term on the right hand side in~\eqref{eq:frw2} is the most involved. We get
\begin{equation}\label{eq:gradp2p3v}
	\begin{aligned}
		&\|(\nb -\ii {\bf A})(c_1\p_2 v + c_2\p_3 v)\|^2_{L^2(\Omg;\dC^3)}\\
		&\qquad =
		|c_1|^2\|(\nb -\ii {\bf A})\p_2 v\|^2_{L^2(\Omg;\dC^3)} + 2\Re\big\{(
		(\nb-\ii{\bf A})c_1\p_2 v, (\nb-\ii{\bf A})c_2\p_3 v)_{L^2(\Omg;\dC^3)} \big\}\\
		&\qquad\qquad +
		|c_2|^2\|(\nb -\ii {\bf A})\p_3 v\|^2_{L^2(\Omg;\dC^3)}.  
	\end{aligned}
\end{equation}
For any face $\G\subset\p\Omg$ of the polyhedral domain $\Omg$ we derive from $v|_\G = 0$ that
$\nb v|_{\G} = \big(\p_\nu v|_{\G}\big)\nu$.
Hence, we get that $\p_i v|_\G = \nu_i \p_\nu v|_{\G}$, $i\in\{1,2,3\}$, where $\nu = (\nu_1,\nu_2,\nu_3)^\top$.
In particular, we get for any $i\in\{2,3\}$ using integration by parts
\begin{align}
	\notag 2\Re\left(\ii\int_\Omg x_1\p_2v \ov{\p_{ii} v}\dd x\right)& =
	-2\Re\left(
	\ii
	\int_\Omg x_1\p_{2i} v\ov{\p_i v}\dd x\right) + 2\Re\left(\int_{\p\Omg} \ii  x_1 \nu_i(\p_2 v|_{\p\Omg})\ov{(\p_i v|_{\p\Omg})} \dd \s\right)\\
	\notag &
	=-2\Re\left(
	\ii
	\int_\Omg x_1\p_{2i} v\ov{\p_i v}\dd x\right) + \Re\left(\int_{\p\Omg} \ii  x_1 \nu_i^2 \nu_2\big|\p_\nu v|_{\p\Omg}\big|^2 \dd \s\right)\\
	\notag &=-2\Re\left(
	\ii
	\int_\Omg x_1\p_{2i} v\ov{\p_i v}\dd x\right)\\
	& =\Re\left(
	\ii
	\int_\Omg x_1\big(\p_{i} v\ov{\p_{2i} v}-\p_{2i} v\ov{\p_i v}\big)\dd x\right).
	\label{eq:Reidentity}
\end{align}	
Using that $v$ is an eigenfunction of $\sfH_{b,\rm D}^\Omg$ corresponding to the eigenvalue $\lm$ and
performing the integration by parts with the help of~\eqref{eq:Reidentity} and Corollary~\ref{cor:byparts} we get for $i\in\{2,3\}$
\begin{align}
	\notag\lm\int_\Omg |\p_i v|^2\dd x &=-\lm\int_\Omg v \ov{\p_{ii} v}\dd x\\
	\notag &=
	\int_\Omg (\nb-\ii {\bf A})^2v \ov{\p_{ii} v}\dd x\\
	\notag& = \Re 
	\int_\Omg (\nb-\ii {\bf A})^2v \ov{\p_{ii} v}\dd x\\
	\notag&= \Re
	\int_\Omg  \big(\p_{11}v + \p_{22} v + \p_{33} v - 2\ii bx_1\p_2 v - b^2x_1^2 v\big) \ov{\p_{ii}v}\dd x\\
	\notag &= \Re 
	\int_\Omg  \big(\p_{1i} v\ov{\p_{1i} v}
	+\p_{2i}v\ov{\p_{2i} v}+
	\p_{3i}v\ov{\p_{3i} v}- 2\ii b x_1  \p_2 v \ov{\p_{ii}v} +  b^2 x_1^2 \p_iv \ov{\p_i v}\big) \dd x\\
	\notag& = \Re\int_\Omg  \big(|\p_{1i} v|^2
	+|\p_{2i} v|^2\!+\!
	|\p_{3i} v|^2- \ii b x_1  \p_{i} v \ov{\p_{2i}v} \!+\! \ii b x_1\p_{2i} v\ov{\p_i v}  +  b^2 x_1^2 |\p_i v|^2\big) \dd x\\
	& = \Re\big((\nb-\ii{\bf A}) \p_i v, (\nb -\ii\bA)\p_i v\big)_{L^2(\Omg;\dC^3)} = \|(\nb -\ii\bA) \p_iv\|^2_{L^2(\Omg;\dC^3)}.\label{eq:piv}
\end{align}	
Moreover, we find via the integration by parts formula~\eqref{eq:integrationbyparts} 
and using  identity~\eqref{eq:p2p3v}
\begin{equation}\label{eq:p2p3vbyparts}
\begin{aligned}
	&\big((\nb-\ii{\bf A})\p_2 v, (\nb-\ii{\bf A})\p_3 v\big)_{L^2(\Omg;\dC^3)} \\
	&\qquad =
	\big(\p_2 v, -(\nb-\ii{\bf A})^2\p_3 v\big)_{L^2(\Omg)}
	+ \big\langle \p_2 v|_{\p\Omg}, \p_\nu(\p_3 v)|_{\p\Omg}
	-\ii \nu\cdot(\bA|_{\p\Omg}) \p_3 v|_{\p\Omg}\big\rangle_{\frac12,-\frac12}\\
	&\qquad  
	=
	\lm\big(\p_2 v, \p_3 v\big)_{L^2(\Omg)}
	+ \big\langle \p_2 v|_{\p\Omg}, \p_\nu(\p_3 v)|_{\p\Omg}
	-\ii \nu\cdot(\bA|_{\p\Omg}) \p_3 v|_{\p\Omg}\big\rangle_{\frac12,-\frac12}.
\end{aligned}	
\end{equation}
Next, our aim is to show that under the imposed symmetry assumption the duality product on the right hand side vanishes. For the function $\hat{v} = v\circ \sfJ\in H^2(\Omg)\cap H^1_0(\Omg)$ constructed in Step 1 we get
\[
\begin{aligned}
	&\big\langle \p_2 v|_{\p\Omg}, \p_\nu(\p_3 v)|_{\p\Omg}
	-\ii \nu\cdot(\bA|_{\p\Omg}) \p_3 v|_{\p\Omg}\big\rangle_{\frac12,-\frac12}\\
	&\qquad
	=\big\langle \p_2 \hat{v}|_{\p\Omg}, \p_\nu(\p_3 \hat{v})|_{\p\Omg}
	-\ii \nu\cdot(\bA|_{\p\Omg}) \p_3 \hat{v}|_{\p\Omg}\big\rangle_{\frac12,-\frac12}\\
	&\qquad =
	((\nb -\ii\bA)\p_2\hat{v},(\nb - \ii\bA)\p_3\hat{v})_{L^2(\Omg;\dC^3)} - \lm(\p_2\hat{v},\p_3\hat{v})_{L^2(\Omg)}\\
	&\qquad
	=
	-
	\big(((\nb -\ii\bA)\p_2v)\circ\sfJ,
	((\nb - \ii\bA)\p_3 v)\circ\sfJ\big)_{L^2(\Omg;\dC^3)}
	+
	\lm
	\big((\p_2v)\circ\sfJ,
	(\p_3 v)\circ\sfJ\big)_{L^2(\Omg)}\\
	&\qquad =
	-
	\big(((\nb -\ii\bA)\p_2v,
	((\nb - \ii\bA)\p_3 v\big)_{L^2(\Omg;\dC^3)}
	+
	\lm
	\big(\p_2v,
	\p_3 v\big)_{L^2(\Omg)}\\
	&\qquad =-\big\langle \p_2 v|_{\p\Omg}, \p_\nu(\p_3 v)|_{\p\Omg}
	-\ii \nu\cdot(\bA|_{\p\Omg}) \p_3 v|_{\p\Omg}\big\rangle_{\frac12,-\frac12} = 0,
\end{aligned}
\]
where in the first step we used $\hat{v} =zv$ with some $z\in\dC$ such that $|z|=1$, in the second step we applied the formula~\eqref{eq:p2p3vbyparts}, in the third step we used that $\hat{v} = v\circ\sfJ$, in the fourth 
step we implicitly used the change of variables $x' = \sfJ x$ and used the symmetry of $\Omg$, in the fifth step we again applied the formula~\eqref{eq:p2p3vbyparts} in the opposite direction. Finally, the whole expression vanished, since we have shown that the duality product is equal to minus itself. 
Hence, we end up with 
the identity
\begin{equation}\label{eq:p23v}
	\Re\big\{\big((\nb - \ii \bA)c_1\p_2 v,(\nb - \ii\bA)c_2\p_3 v\big)_{L^2(\Omg;\dC^3)}\big\} = 
	\lm\Re (c_1\p_2 v, c_2\p_3 v)_{L^2(\Omg)}
\end{equation}
where we implicitly used that $(\p_2v,\p_3 v)_{L^2(\Omg)}
= -(v,\p_{23}v)_{L^2(\Omg)} = -(\p_{23}v,v)_{L^2(\Omg)}$ is a real number.
Substituting~\eqref{eq:piv} and~\eqref{eq:p23v} into~\eqref{eq:gradp2p3v} we arrive at
\begin{equation}\label{eq:third3}
\begin{aligned}
&\|(\nb -\ii {\bf A})(c_1\p_2 v + c_2\p_3 v)\|^2_{L^2(\Omg;\dC^3)}\\
&\qquad= \lm|c_1|^2\|\p_2 v\|^2_{L^2(\Omg)}
+ 2\lm\Re\big\{(c_1\p_2 v,c_2\p_3 v)_{L^2(\Omg)}\big\}
+
\lm |c_2|^2\|\p_3 v\|^2_{L^2(\Omg)}\\
&\qquad = \lm\|c_1\p_2 v + c_2\p_3 v\|^2_{L^2(\Omg)}.
\end{aligned} 
\end{equation}
Finally, we get
from~\eqref{eq:first2},~\eqref{eq:second2} and~\eqref{eq:third3} that
\[
\begin{aligned}
	\frh_{b,\rm N}^\Omg[w]&\le \lm\|u\|^2_{L^2(\Omg)} + 
	2\lm\Re\big\{(u,c_1\p_2 v +c_2\p_3v)_{L^2(\Omg)}\big\} + \lm\|c_1\p_2 v+c_2\p_3 v\|^2_{L^2(\Omg)}\\
	& = \lm\|u+c_1\p_2 v+c_2\p_3 v\|^2_{L^2(\Omg)} = \lm\|w\|^2_{L^2(\Omg)}.
\end{aligned}	
\]
We have shown that for any $w\in\cM\subset H^1(\Omg)$ it holds that $\frh_{b,\rm N}^\Omg[w]\le \lm\|w\|^2_{L^2(\Omg)}$. Since $\cM$ is a $(k+2)$-dimensional linear subspace of $H^1(\Omg)$, we get by the min-max principle that $\mu_{k+2}^b(\Omg)\le \lm = \lm_k^b(\Omg)$. Thus, the inequality in the theorem is proved for convex polyhedral domains invariant under the mapping $\sfJ$.

\smallskip

\noindent{\it Step 4: general convex domains invariant under the mapping $\sfJ$.}
For any bounded convex domain $\Omg\subset\dR^3$ one can construct by~\cite{Br08} a sequence of convex polyhedral
domains $\{\Omg_n\}_{n\in\dN}$ such that the inclusion $\Omg\subset\Omg_n$ holds for all $n\in\dN$ and that $|\Omg_n\sm\Omg|\arr 0$. If the domain $\Omg$ satisfies, in addition, $\Omg = \sfJ(\Omg)$, we can consider the sequence of convex polyhedral domains $\wh{\Omg}_n := \Omg_n\cap\sfJ(\Omg_n)$, which clearly satisfies both conditions $\Omg\subset\wh\Omg_n$ for all $n\in\dN$ and $|\wh\Omg_n\sm\Omg|\arr 0$ and, in addition, we have $\wh\Omg_n = \sfJ(\wh\Omg_n)$.

Let us fix $k\in\dN$.
The min-max principle for the magnetic Laplacian with Dirichlet boundary conditions~\eqref{eq:minmaxD} and the inclusion $\Omg\subset\wh\Omg_n$ immediately imply that
\begin{equation}\label{eq:incl_ineq12}
	\lm_k^b(\Omg) \ge \lm_k^b(\wh\Omg_n),\qquad \text{for all}\,\,n\in\dN.
\end{equation}  
It follows from Proposition~\ref{prop:semicont} in Appendix~\ref{app} that
\begin{equation}\label{eq:incl_ineq22}
	\liminf_{n\arr\infty}\mu_{k+2}^b(\wh\Omg_n) \ge \mu_{k+2}^b(\Omg).
\end{equation}
Combining~\eqref{eq:incl_ineq12} and~\eqref{eq:incl_ineq22} with the eigenvalue inequality in the theorem already shown for convex polygons, we get
\[
\mu_{k+2}^b(\Omg) \le \liminf_{n\arr\infty}\mu_{k+2}^b(\wh\Omg_n) \le
\liminf_{n\arr\infty}\lm_{k}^b(\wh\Omg_n) \le \lm_k^b(\Omg),
\]
by which the inequality in the theorem is proved also for general convex domains invariant under the mapping $\sfJ$.
\section{The case of the disk}
\label{sec:disk}
In this section, we will consider the magnetic Laplacian on the disk. Recall that the Levine-Weinberger inequality in two dimensions reads as follows: the $(k+2)$-th Neumann eigenvalue of the Laplacian on a planar convex domain is not larger than its $k$-th Dirichlet eigenvalue. The index $k+2$ can not be replaced by $k+3$, because the fourth Neumann eigenvalue of the Laplacian on the disk is larger than its first Dirichlet eigenvalue. Thus, it is natural to test the magnetic eigenvalue inequality in Theorem~\ref{thm:main1} also on the disk.

Let $\cB\subset\dR^2$ be the disk of unit radius centred at the origin. In view of the scaling properties $\lm_k^b(t\cB) = t^{-2}\lm_k^{bt^2}(\cB)$
and $\mu_k^b(t\cB) = t^{-2}\mu_k^{bt^2}(\cB)$
valid for any $k\in\dN$ and all $t> 0$ (see~\cite[p. 456]{CLPS23}), it suffices to test the eigenvalue inequality on the unit disk.
The eigenvalues of $\sfH_{b,\rm N}^\cB$ and $\sfH_{b,\rm D}^\cB$ are continuous functions in $b$ (cf.~\cite[Appendix B]{KL22} for the Neumann case, the Dirichlet case is analogous).
Hence, we know that $\mu_4^b(\cB) > \lm_1^b(\cB)$
for all sufficiently small $b> 0$, because the same inequality holds for the usual Laplacian, which corresponds to $b = 0$.
Thus, the index $k+1$ in Theorem~\ref{thm:main1} can not be, in general, replaced by $k+3$. However, this observation does not exclude the possibility that the index $k+1$ can be replaced by $k+2$.

It follows from~\cite[Proposition 2.1]{FFGKS23} that 
$\mu_3^b(\cB) < b$ if and only if $b > 4$,
because the flux $\Phi := \frac{b|\cB|}{2\pi} = \frac{b}{2} > 2$ for $b > 4$. On the other hand, according to~\cite[Lemma 1.4.1]{FH} we have $\lm_1^b(\cB) \ge b$ for all $b > 0$. Hence, we get that $\mu_3^b(\cB) < \lm_1^b(\cB)$ for all $b > 4$. 
In the remaining part of this appendix, we will analyse the case $b\in(0, 4]$ numerically. The numerics is rather simple and reduces to solving explicit transcendental equations in terms of generalized Laguerre functions.

The magnetic Neumann Laplacian $\sfH_{b,\rm N}^{\cB}$ can be decomposed by \eg~\cite[Section 3]{KL22} up to unitary equivalence as
\begin{equation}\label{eq:orth}
	\sfH_{b,\rm N}^{\cB} \simeq \bigoplus_{n\in\dZ} \sfT_{b,n},
\end{equation}
where the self-adjoint fiber operator $\sfT_{b,n}$, $n\in\dZ$, acts in the Hilbert space $L^2((0,1);r\dd r)$ and is associated with the quadratic form
\begin{equation}\label{eq:fiber}
\begin{aligned}
	\frt_{b,n}[f]& := \int_0^1|f'(r)|^2r\dd r +\int_0^1\left(\frac{n}{r}-\frac{br}{2}\right)^2|f(r)|^2r\dd r,\\
	\qquad \dom\frt_{b,n} &:= \big\{f\colon f,f', nfr^{-1}\in L^2((0,1);r\dd r)\big\}.
\end{aligned}	
\end{equation}
Let us denote by $\mu_{n,1} = \mu_{n,1}(b)$, $n\in\dZ$, the lowest eigenvalue of the fiber operator $\sfT_{b,n}$.
In view of the decomposition~\eqref{eq:orth} for any $\lm > 0$
\begin{equation}
\#\{k\in\dN\colon\mu_k^b(\cB) \le \lm\} \ge \#\{n\in\dZ\colon \mu_{n,1}(b) \le \lm\}.   
\end{equation}
Let us denote by $L^a_\gamma$ the generalized Laguerre function (see, \eg, ~\cite[\S 22]{AS} for precise definitions and properties).
By the analysis in \cite[Appendix B]{CLPS23},
$\mu_{n,1}$ is the smallest positive solution of the equation
\begin{equation}\label{key}
	F_{n,b}(\mu) := \frac{\dd}{\dd r}\left(e^{-\frac{br^2}{4}} r^n L^n_{\frac12\big(\frac{\mu}{b}-1\big)}\left(\frac{r^2b}{2}\right)\right)\bigg|_{r=1} = 0,
\end{equation}
and the associated eigenfunction is given by
\[
	v_n(r) =  
	e^{-\frac{br^2}{4}} r^n L^n_{\frac12\big(\frac{\mu_{n,1}}{b}-1\big)}\left(\frac{r^2b}{2}\right).
\]
For the magnetic Dirichlet Laplacian on the disk $\sfH_{b,{\rm D}}^\cB$ we know according to~\cite[Proposition 2.1]{E96} that  $\lm_1^b(\cB)$ is equal to the lowest eigenvalue $\lm_{0,1} = \lm_{0,1}(b)$ of the one-dimensional self-adjoint fiber operator in $L^2((0,1);r\dd r)$
associated with the quadratic form
\begin{equation*}
\begin{aligned}
	\frt_b[f] &:= \int_0^1|f'(r)|^2r\dd r +\int_0^1\frac{b^2r^3}{4}|f(r)|^2\dd r\\
	\dom\frt_{b} &:= \big\{f\colon f,f'\in L^2((0,1);r\dd r), \, f(1) = 0\big\}.
\end{aligned}	
\end{equation*}
This is a direct consequence of the fact that, in the symmetric gauge of the homogeneous magnetic field, the ground state of the magnetic Laplacian on the disk is a radial function.
Mimicking the analysis in~\cite[Appendix B]{CLPS23} we get that
$\lm_{0,1}$ is the smallest positive solution of the equation
\begin{equation}\label{key}
	G_b(\lm) := L^0_{\frac12\big(\frac{\lm}{b}-1\big)}\left(\frac{b}{2}\right)= 0.
\end{equation}
In Figure~\ref{fig}, we have plotted the graphs of $\lm_{0,1}$, $\mu_{0,1}$, $\mu_{1,1}$, $\mu_{-1,1}$, $\mu_{2,1}$ as functions of $b$. In the plotted range of $b$, for any $b >0$ among 
the values 
$\mu_{0,1}(b)$, $\mu_{1,1}(b)$, $\mu_{-1,1}(b)$, $\mu_{2,1}(b)$ at least three do not exceed $\lm_{0,1}(b)$. We also remark that at $b = 2$ the graphs of $\lm_{0,1}$, $\mu_{-1,1}$, and $\mu_{2,1}$ indeed cross at one point. This property can be rigorously checked.
By~\cite[8.970 (4)]{GR} we have the following explicit expression for the generalized Laguerre functions: $L^0_1(x) = 1-x$, $L^{-1}_1(x) = -x$, and $L^2_1(x) = 3-x$. Using these expression we find
\[
\begin{aligned}
	G_2(6) &= L^0_1(1) = 0,\\
	F_{-1,2}(6) &= \frac{\dd}{\dd r}\left(e^{-\frac{r^2}{2}}
		r^{-1} L_1^{-1}(r^2)\right)\bigg|_{r=1}
		=
		-
		\frac{\dd}{\dd r}\left(e^{-\frac{r^2}{2}}
		r\right)\bigg|_{r=1}
		= \left(e^{-\frac{r^2}{2}}\big[r^2-1\big]\right)\bigg|_{r=1}  =0,\\
	F_{2,2}(6) &=	\frac{\dd}{\dd r}\left(e^{-\frac{r^2}{2}}
	r^2 L_1^{2}(r^2)\right)\bigg|_{r=1}
	=\frac{\dd}{\dd r}\left(e^{-\frac{r^2}{2}}
	r^2 (3-r^2)\right)\bigg|_{r=1}\\
	&=
	\left(e^{-\frac{r^2}{2}}\big[r^5-7r^3+6r\big]\right)\bigg|_{r=1} = 0.
\end{aligned}	
\]
\begin{figure}
\includegraphics[width=9cm]{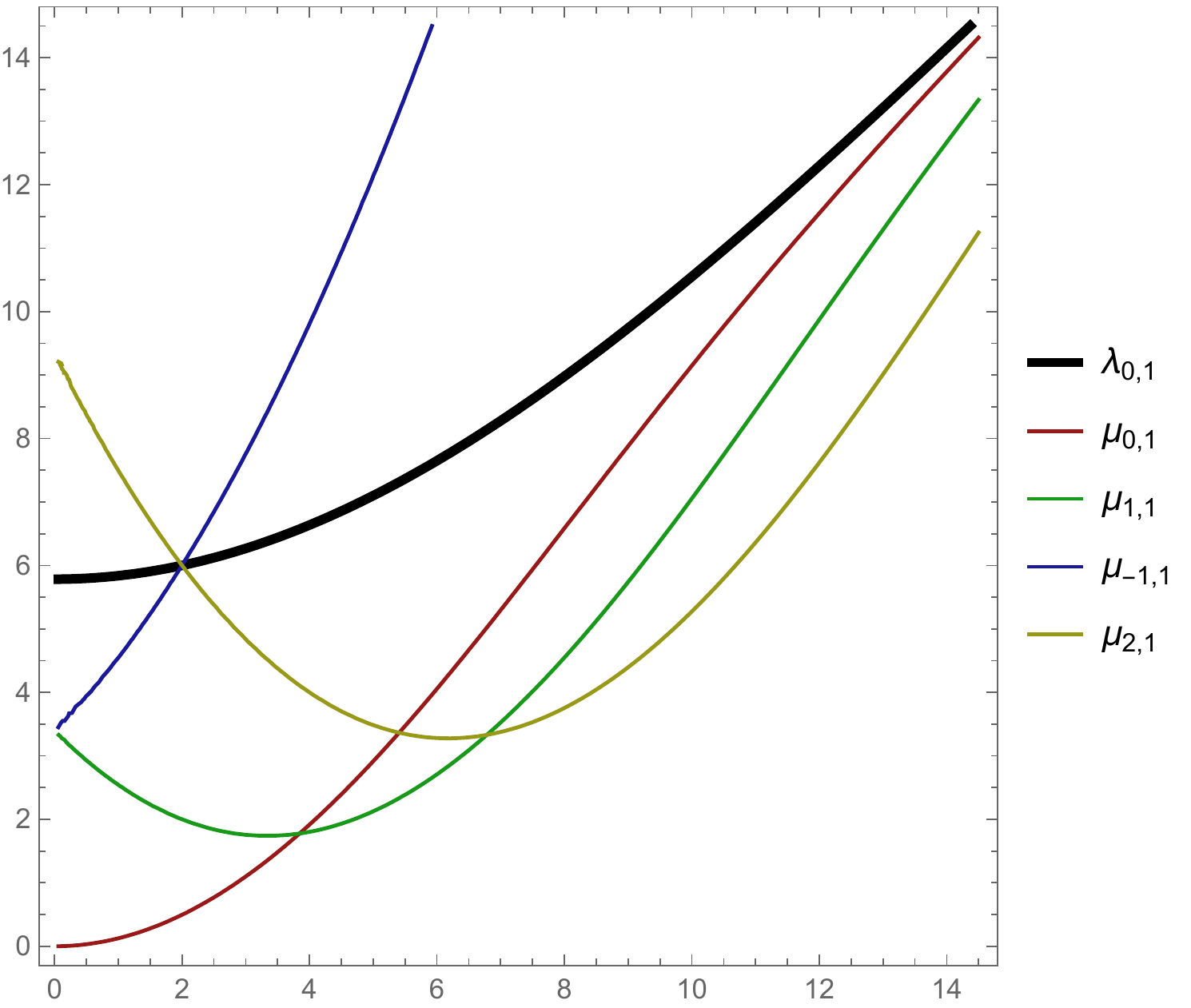}
\caption{The lowest eigenvalues of the fiber operators are plotted as functions of $b$: $\lm_{0,1}$ --  thick black curve; $\mu_{0,1}$ --   thin red curve; $\mu_{1,1}$ -- thin green curve; $\mu_{-1,1}$ --  thin blue curve; $\mu_{2,1}$ --  thin yellow curve.  }
\label{fig}
\end{figure}%
This numerical test provides an evidence for the validity of the inequality $\mu_3^b(\cB)\le \lm_1^b(\cB)$ for all $b > 0$.
Therefore, a counterexample showing sharpness of Theorem~\ref{thm:main1} should either be based on the comparison of higher eigenvalues of the disk or should be based on a domain different from the disk. Thus, it remains an open problem whether
the inequality $\mu_{k+2}^b(\Omg)\le \lm_k^b(\Omg)$
for all $b > 0$ and any $k\in\dN$ holds for any convex domain $\Omg\subset\dR^2$. We would like to refrain from making any conjectures.  
\section*{Acknowledgement}
The author acknowledges the support by the Czech Science Foundation (GA\v{C}R) within the project 21-07129S.
\begin{appendix}
\section{Lower semicontinuity of magnetic Neumann eigenvalues}\label{app}
In this appendix, we establish a lower semi-continuity result for the magnetic Neumann eigenvalues on varying convex domains satisfying an inclusion property.  Although, we expect that under our assumptions the magnetic Neumann eigenvalues are in fact continuous with respect to the variation of the domain, lower semi-continuity turns out to be sufficient for our purposes. General stability results for the eigenvalues of
elliptic operators with Neumann-type boundary conditions under variations of the domain can be found \eg in~\cite{BD02, BL07}, see also the references therein. We expect that the abstract results in~\cite{BL07} can be adjusted to the magnetic Neumann Laplacian, but we prefer to perform a more explicit analysis.

In the proof of a lower semi-continuity result the following lemma is useful.
\begin{lem}[{\cite[Lemma 4.9, Remark 4.2]{MKh}}]
	\label{lem:MKh}
	Let $U,V\subset\dR^d$, $d\ge 2$, be bounded convex domains such that $U\subset V$. Then for any $u\in H^1(V)$ the following inequality holds
	\[
		\int_{V\sm U}|u|^2\dd x \le \frac{2|V\sm U|}{|V|}\int_V|u|^2\dd x + \frac{C(d)({\rm diam}\,(V))^{d+1}|V\sm U|^{1/d}}{|V|}\int_V|\nb u|^2\dd x,
	\]
	where ${\rm diam}\,(V) > 0$ is the diameter of $V$ and the constant $C(d) > 0$ depends only on the dimension.
\end{lem}
Now, we are ready to formulate and prove the key result of this appendix.
\begin{prop}\label{prop:semicont}
	Let $\Omg\subset\dR^d$ with $d =2,3$ be a bounded convex domain. Let $\{\Omg_n\}_{n\in\dN}$ be a sequence of bounded convex domains in $\dR^d$ such that $\Omg_n\supset\Omg$ for all $n\in\dN$ and that $|\Omg_n\sm\Omg|\arr 0$. Then, for all $b > 0$ and any $k\in\dN$
	\[
		\liminf_{n\arr\infty} \mu_k^b(\Omg_n)\ge \mu_k^b(\Omg).
	\] 
\end{prop}
\begin{proof}
	For the convenience of the reader, we split the argument into two steps. In the first step, we obtain estimates of the eigenfunctions of $\sfH_{b,{\rm N}}^{\Omg_n}$ inside $\Omg_n\sm\Omg$. Then, in the second step we apply the min-max principle to $\sfH_{b,{\rm N}}^\Omg$ using
	as trial functions the restrictions to $\Omg$ of the eigenfunctions of $\sfH_{b,{\rm N}}^{\Omg_n}$. This strategy is reminiscent of the one used by Levine and Weinberger in~\cite[Lemma 4.6]{LW86}. 	
	
	\smallskip
	
	\noindent {\it Step 1: estimates of the eigenfunctions in $\Omg_n\sm\Omg$.}
	Let us fix $k\in\dN$. Let $u_{1,n},u_{2,n},\dots, u_{k,n}\in H^1(\Omg_n)$ be the orthonormal eigenfunctions of $\sfH_{b,\rm N}^{\Omg_n}$ corresponding to the
	lowest $k$ eigenvalues: $\mu_1^b(\Omg_n),\mu_2^b(\Omg_n),\dots, \mu_k^b(\Omg_n)$.
	Note that the geometric assumptions on $\Omg_n$ and $\Omg$ yield 
	\[
		\rho:=\sup_{n\in\dN} {\rm diam}\,(\Omg_n) < \infty
		\qquad\aa :=\sup_{n\in\dN}|\Omg_n\sm\Omg| <\infty.
	\]
	Combining Lemma~\ref{lem:MKh} and the diamagnetic inequality~\cite[Theorem 7.21]{LL} we get for any $l\in\{1,2,\dots,k\}$
	\[
	\begin{aligned}
		\int_{\Omg_n\sm\Omg}|u_{l,n}|^2\dd x & \le \frac{2|\Omg_n\sm\Omg|}{|\Omg_n|}\int_{\Omg_n}|u_{l,n}|^2\dd x+ \frac{C(d)({\rm diam}\,(\Omg_n))^{d+1}|\Omg_n\sm\Omg|^{1/d}}{|\Omg_n|}
		\int_{\Omg_n}|\nb |u_{l,n}||^2\dd x \\
		&\le
		\frac{2|\Omg_n\sm\Omg|}{|\Omg|}
		+\frac{C(d)\rho^{d+1}|\Omg_n\sm\Omg|^{1/d}}{|\Omg|}
		\int_{\Omg_n}|(\nb - \ii{\bf A})u_{l,n}|^2 \dd x\\
		&\le
		\frac{1}{|\Omg|}
		\left(2|\Omg_n\sm\Omg|
		+C(d)\rho^{d+1}|\Omg_n\sm\Omg|^{1/d}
		\mu_l^b(\Omg_n)\right).
	\end{aligned}	
	\]
	Applying the min-max principles~\eqref{eq:minmax} to
	the operators $\sfH_{b,\rm N}^{\Omg_n}$, $\sfH_{b,\rm D}^{\Omg_n}$, and $\sfH_{b,\rm D}^\Omg$ we get for $l\in\{1,2,\dots,k\}$ the following chain of inequalities
	\[
		\mu_l^b(\Omg_n) \le \mu_k^b(\Omg_n)\le \lm_k^b(\Omg_n)\le \lm_k^b(\Omg),
	\]
	where for the last inequality we employed $\Omg\subset\Omg_n$.	Thus, there exists a constant $c > 0$, which depends only on $\Omg$, $\aa$, $\rho$, $b$, and $k$ such that
	\[
		\int_{\Omg_n\sm\Omg}|u_{l,n}|^2\dd x
		\le c|\Omg_n\sm\Omg|^{1/d} 
	\]
	for all $l \in\{1,2,\dots,k\}$.
	In particular, for any $\eps > 0$ there exists
	a sufficiently large $N\in\dN$ such that for
	any $l\in\{1,2,\dots,k\}$ and all $n\ge N$
	\begin{equation}\label{eq:uln}
		\int_{\Omg_n\sm\Omg}|u_{l,n}|^2\dd x\le \eps.
	\end{equation}
	\noindent{\it Step 2: application of the min-max principle.}	
	Consider the subspace 
	\[
	\cL_{k,n} := {\rm span}\,\{u_{1,n}|_\Omg, u_{2,n}|_\Omg,\dots, u_{k,n}|_\Omg\}\subset H^1(\Omg).
	\]
	For any $v = \sum_{l=1}^k c_l u_{l,n}|_{\Omg} \in\cL_{k,n}$
	with arbitrary $c_l\in\dC$ ($l=1,2,\dots,k$) we get
	\begin{equation}\label{eq:vest}
	\begin{aligned}
		\int_\Omg|v|^2\dd x &=
		\sum_{i=1}^k\sum_{j=1}^k c_i\ov{c_j} \int_\Omg u_{i,n}\ov{u_{j,n}}\dd x \\
		&\ge
		\sum_{l=1}^k |c_l|^2\int_\Omg|u_{l,n}|^2\dd x-
		\sum_{\begin{smallmatrix}i,j=1\\
				i\ne j\end{smallmatrix}}^k|c_i||c_j|
			\left|\int_{\Omg_n\sm\Omg} u_{i,n}\ov{u_{j,n}}\dd x\right|\ge \sum_{l=1}^k|c_l|^2(1-k\eps).
	\end{aligned}	
	\end{equation}
	Hence, for all sufficiently small $\eps > 0$ we get that $v\ne 0$ provided that not all of the coefficients $\{c_l\}_{l=1}^k$ are equal to zero. Hence, it holds that $\dim\cL_{k,n} = k$ for all sufficiently large $n\in\dN$. 
			
	By~\eqref{eq:minmaxN} we conclude that
	there exist complex coefficients $\{c_l'\}_{l=1}^k$ such that for $v = \sum_{l=1}^k c_l' u_{l,n}|_\Omg\in H^1(\Omg)$
	one has
	\[
		\mu_k^b(\Omg) \le 
		\frac{\displaystyle\int_{\Omg}|(\nb -\ii {\bf A}) v|^2\dd x}{\displaystyle \int_\Omg |v|^2 \dd x}.
	\]
	Let us define the function $u = \sum_{l=1}^k c_l' u_{l,n}\in H^1(\Omg_n)$.
	Choosing $N\in\dN$ such that the estimate in~\eqref{eq:uln} holds with $\eps < \frac{1}{k}$ for all $n\ge N$ and applying the lower bound in~\eqref{eq:vest} we end up with 
	\[  
		\mu_k^b(\Omg)\le \frac{1}{1-k\eps}
		\frac{\displaystyle\int_{\Omg_n}|(\nb -\ii {\bf A}) u|^2\dd x}{\displaystyle \int_{\Omg_n} |u|^2 \dd x}
		\le \frac{\mu_k^b(\Omg_n)}{1-k\eps}.
	\]
	From this inequality we get
	\[
		\liminf_{n\arr\infty}\mu_k^b(\Omg_n) \ge \mu_k^b(\Omg)(1-k\eps).
	\]
	The claim follows from the fact that we can take $\eps > 0$ arbitrarily small.
\end{proof}
\end{appendix}
\newcommand{\etalchar}[1]{$^{#1}$}

\end{document}

%% file: commands.tex
\newcommand\nb{\nabla}
\newcommand{\beq}{\begin{equation} \begin{split}}
\newcommand{\eeq}{\end{split} \end{equation}}

\newcommand\Omg{\Omega}

\makeatletter
\def\section{\@startsection{section}{1}\z@{.9\linespacing\@plus\linespacing}%
	{.7\linespacing} {\fontsize{13}{14}\selectfont\bfseries\centering}}
\def\paragraph{\@startsection{paragraph}{4}%
	\z@{0.3em}{-.5em}%
	{$\bullet$ \ \normalfont\itshape}}

\@addtoreset{equation}{section}
\makeatother

\renewcommand\and{\qquad\text{and}\qquad}

\newcommand\sm{\setminus}

\newcommand{\comm}[1]{}

\def\sfH{\mathsf{H}}

\def\bm1{\mathbbm{1}}
\def\G{\Gamma}

\def\s{\sigma}

\def\p{\partial}

\def\omg{\omega}

\def\Re{{\rm Re}\,}
\def\Im{{\rm Im}\,}

\def\arr{\rightarrow}

\newcommand{\sfJ}{\mathsf{J}}

\def\aa{\alpha}
\def\lm{\lambda}

\def\s{\sigma}

\def\ii{{\mathsf{i}}}
\def\p{\partial}

\def\sfH{\mathsf{H}}

\def\dd{{\,\mathrm{d}}}

\def\omg{\omega}

\newcounter{counter_a}
\newenvironment{myenum}{\begin{list}{{\rm(\roman{counter_a})}}%
{\usecounter{counter_a}
\setlength{\itemsep}{1.ex}\setlength{\topsep}{0.8ex}
\setlength{\leftmargin}{5ex}\setlength{\labelwidth}{5ex}}}{\end{list}}

\newcommand{\eg}{{\it e.g.}\,}

\newcommand{\cf}{{\it cf.}\,}

\numberwithin{figure}{section}
\numberwithin{equation}{section}
\theoremstyle{plain}
\newtheorem*{thm*}{Theorem}
\newtheorem{thm}{Theorem}[section]

\newtheorem{lem}[thm]{Lemma}
\newtheorem{prop}[thm]{Proposition}

\newtheorem{cor}[thm]{Corollary}

\newtheorem{dfn}[thm]{Definition}
\theoremstyle{remark}
\newtheorem{remark}[thm]{Remark}
\theoremstyle{plain}

\newcommand{\beu}{\begin{equation*}}
\newcommand{\eeu}{\end{equation*}}
\newcommand{\besu}{\begin{equation*}
\begin{aligned}}
\newcommand{\eesu}{\end{aligned}
\end{equation*}}
\newcommand{\bes}{\begin{equation}
\begin{aligned}}
\newcommand{\ees}{\end{aligned}
\end{equation}}

\newcommand\cB{\mathcal B}

\newcommand\cL{\mathcal L}
\newcommand\cM{\mathcal M}

\newcommand\frh{\mathfrak h}

\newcommand\eps{\varepsilon}

\newcommand\ov{\overline}

\newcommand\wh{\widehat}

\newcommand\void[1]{}

\def\ov{\overline}
\def\eps{\varepsilon}

\def\frt{{\mathfrak t}}

      \def\dC{{\mathbb C}}

   \def\dN{{\mathbb N}}   
      \def\dR{{\mathbb R}}

   \def\dZ{{\mathbb Z}}

   \def\sfH{{\mathsf H}}   
\def\sfJ{{\mathsf J}}

   \def\sfT{{\mathsf T}}

   \def\cB{{\mathcal B}}

      \def\cL{{\mathcal L}}
\def\cM{{\mathcal M}}

\newcommand{\dom}{\mathrm{dom}\,}